\documentclass{amsart}

\usepackage{amssymb,color}
\usepackage{amsfonts}
\usepackage{amsmath}
\usepackage{euscript}
\usepackage{enumerate}
\usepackage{graphics}

\newtheorem{theorem}{Theorem}[section]
\newtheorem{lemma}[theorem]{Lemma}

\newtheorem{cor}[theorem]{Corollary}

\newtheorem*{Theorem1'}{Theorem 1'}

\theoremstyle{definition}

\theoremstyle{remark}

\numberwithin{equation}{section}

\newcommand  \s{\sigma}

\newcommand \Z{{\mathbb Z}}

\newcommand \al{{\alpha}}
\newcommand \be{{\beta}}

\newcommand \si{{\sigma}}

\begin{document}

\title[On the splitting ring of a polynomial] {\small On the splitting ring of a polynomial}

\author{Fernando Szechtman}
\address{Department of Mathematics and Statistics, University of Regina, Canada}
\email{fernando.szechtman@gmail.com}
\thanks{The author was supported in part by an NSERC discovery grant}

%    General info

%\date{January 1, 2001 and, in revised form, June 22, 2001.}



\begin{abstract} Let $f(Z)=Z^n-a_{1}Z^{n-1}+\cdots+(-1)^{n-1}a_{n-1}Z+(-1)^na_n$ be a monic polynomial with coefficients
in a ring~$R$ with identity, not necessarily commutative.
We study the ideal $I_f$ of $R[X_1,\dots,X_n]$ generated by
$\s_i(X_1,\dots,X_n)-a_{i}$, where $\s_1,\dots,\s_n$ are the
elementary symmetric polynomials, as well as the quotient ring
$R[X_1,\dots,X_n]/I_f$.
\end{abstract}

\maketitle

\section{Introduction}

%It will be convenient to state our results completely deprived of any preamble.

Let $F$ be a field and let $f(Z)\in F[Z]$ be a polynomial of degree
$n\geq 1$ having distinct roots $r_1,\dots, r_n$ in a splitting
field $K$. Let $F[X_1,\dots,X_n]\to K$ be the epimorphism of
$F$-algebras $p(X_1,\dots,X_n)\to p(r_1,\dots,r_n)$ and let $J_f$ be
its kernel. The Galois group $\mathrm{Gal}(K/F)$ can be identified with
the subgroup of $S_n$ that preserves all algebraic relations
amongst the roots of $f(Z)$, i.e., the subgroup of $S_n$ that
preserves $J_f$.

Let $\s_1,\dots,\s_n\in F[X_1,\dots,X_n]$ be the elementary symmetric polynomials. It is clear that
$$
I_f=(\s_1(X_1,\dots,X_n)-\s_1(r_1,\dots,r_n),\dots,\s_n(X_1,\dots,X_n)-\s_n(r_1,\dots,r_n))
$$
is included in $J_f$, and one verifies that $I_f=J_f$ if and only if $[K:F]=n!$.

Regardless of whether $I_f=J_f$  or not, the quotient algebra
$F[X_1,\dots,X_n]/I_f$ possesses generic features valid in great
generality, and as such has been a classical object of
investigation when $F$ is replaced by a commutative ring with identity.

Let $R$ be a non-zero ring with identity. Given a monic polynomial
$$f(Z)=Z^n-a_{1}Z^{n-1}+a_2
Z^{n-2}+\cdots+(-1)^{n-1}a_{n-1}Z+(-1)^na_n
$$
of degree $n\geq 1$ in $R[Z]$, consider the ideal $I_f$ of $R[X_1,\dots,X_n]$ given~by
$$
I_f=(\s_1(X_1,\dots,X_n)-a_{1},\dots,\s_n(X_1,\dots,X_n)-a_n),
$$
where $\s_1,\dots,\s_n\in
R[X_1,\dots,X_n]$ are the elementary symmetric polynomials, as well as the quotient ring
$$R_f=R[X_1,\dots,X_n]/I_f.$$
We refer to $R_f$ as the universal splitting ring for $f$ over $R$.

Assume until further notice that $R$ is commutative. As far as we
know, the first systematic study of $R_f$ was made by Nagahara
\cite{N}, who showed the following. The ring $R_f$ is a free
$R$-module of rank $n!$ with basis $r_1^{\al_1}\cdots
r_n^{\al_n}$, where $r_i=X_i+I_f$ and $0\leq \al_i\leq n-i$ for
every $1\leq i\leq n$; the composite map $R\to R[X_1,\dots,R_n]\to R_f$ is injective;
$f(Z)=(Z-r_1)\cdots (Z-r_1)$ holds in $R_f[Z]$; the symmetric group $S_n$ acts via automorphisms on $R_f$
with ring of invariants equal to $R$, provided the discriminant
$\delta(f)$ of $f$ is a unit in $R$.

Independently, and shortly afterwards, Barnard \cite{Ba} proved
essentially the same results, although his statement concerning
the ring of invariants was inaccurate.

A few years later Wang \cite{W} established an isomorphism, under
the assumption that $\delta(f)$ be a unit, between $R_f$ and a
ring that Auslander and Goldman \cite{AG} had previously
constructed in a completely different way.

Shortly afterwards $R_f$ matured into book form, described first
by Bourbaki \cite{Bo} and later by Pohst and Zassenhaus \cite{PZ}.

Lately, $R_f$ has attracted considerable attention
following a paper by Ekedahl and Laksov \cite{EL}, who investigate
$R_f$ when $f$ is a generic polynomial (with coefficients
algebraically independent over $R$), make an independent study of
the ring of invariants of $R_f$ under $S_n$, and give applications
of $R_f$ to Galois theory.

More recently, the concept of splitting ring of a polynomial has been generalized to the notion of Galois closure for ring extensions
by Bhargawa and Matthew \cite{BS} as well as Gioia \cite{G}.

We henceforth remove the requirement that $R$ be commutative. Our goal is to study the left regular representation $\ell:R_f\to\mathrm{End}_R(R_f)$,
with the aim of producing linear and matrix realizations of $R_f$, which is viewed here as a right $R$-module.

In order to understand the $R$-linear maps $\ell_{r_i}$, where $$r_i=X_i+I_f\in R_f,$$ a knowledge of the relations amongst $r_1,\dots,r_n$
is required. The defining generators of $I_f$, namely $\si_i-a_i$, are not well suited for this purpose. We consider, instead, the polynomials $$f_1(X_1)\in R[X_1],f_2(X_1,X_2)\in R[X_1,X_2],\dots,f_n(X_1,\dots,X_n)\in R[X_1,\dots,X_n],$$ recursively defined by
\begin{equation}\label{defensor}
f_1(X_1)=f(X_1)
\end{equation}
and
\begin{equation}\label{defensor2}
f_2(X_1,X_2)=\frac{f_1(X_2)-f_1(X_1)}{X_2-X_1},\; f_3(X_1,X_2,X_3)=\frac{f_2(X_1,X_3)-f_2(X_1,X_2)}{X_3-X_2},\dots,
\end{equation}
that is,
\begin{equation}\label{defensor3}
f_{i+1}(X_1,\dots,X_i,X_{i+1})=\frac{f_{i}(X_1,\dots,X_{i-1},X_{i+1})-f_{i}(X_1,\dots,X_{i-1},X_{i})}{X_{i+1}-X_i},
\end{equation}
the quotient of dividing $f_{i}(X_1,\dots,X_{i-1},X_{i+1})-f_{i}(X_1,\dots,X_{i-1},X_{i})$ by~$X_{i+1}-X_i$.

The polynomials $f_1,\dots,f_n$ play a decisive role in the study of $R_f$ and are shown to generate $I_f$.
Moreover, closed formulae are given for $f_1,\dots,f_n$ and
their relationship to $\si_1-a_1,\dots,\si_n-a_n$. Furthermore, $f_1,f_2,\dots,f_n$
are shown to be symmetric in $R[X_1], R[X_1,X_2],\dots, R[X_1,\dots,X_n]$.

Now, it is no longer true that the composite map $\Gamma:R\to R[X_1,\dots,R_n]\to R_f$
is injective. In fact, it is entirely possible for $R_f$ to be zero. This will certainly be the case if $R$ is simple and at
least one of the coefficients of $f$ is not central. In any case, let $L_f$ be the ideal of $R$ generated by all commutators $[x,a_i]=xa_i-a_ix$, where $x\in R$ and $1\leq i\leq n$,
and let $M_f=\ker(\Gamma)$, that is, $M_f=I_f\cap R$. It is clear that $L_f\subseteq M_f$, and we show that equality prevails. Set $T_f=R/L_f$ and let $\pi:R\to T_f$ be the canonical projection, which we extend to a ring epimorphism $R[Z]\to T_f[Z]$, also denoted by~$\pi$. Note that
$R_f$ is naturally a $T_f$-module.

We readily verify that the universal splitting ring for $f$ over $R$ is isomorphic, as ring and $T_f$-module,
to the universal splitting ring for $\pi(f)$ over $T_f$. Note that the coefficients of $\pi(f)$ are central in $T_f$.
Thus, when studying $R_f$, there is no loss of generality in assuming that the coefficients of $f$ are already central in $R$.
This assumption will be kept under further notice. In this context, $\Gamma$ is shown to be injective and, moreover, $R_f$ is seen to be a free $R$-module with basis $r_1^{\al_1}\cdots r_n^{\al_n}$, where $0\leq \al_i\leq n-i$.

We next realize $R_f$ as a ring, say $S_f$, of $R$-linear operators acting on free right $R$-module.
Our construction of $S_f$ is completely independent of $R_f$ and is based solely on the polynomials $f_1,\dots,f_n$.

We also provide a matrix realization of $R_f$. More precisely, we construct matrices $A_1,\dots,A_n\in M_{n!}(R)$ satisfying
the following properties: $A_1,\dots,A_n$ commute with each other and with every element of $R$; $\si_i(A_1,\dots,A_n)=a_i$
for all $1\leq i\leq n$; $A_1^{\alpha_1}\cdots,A_n^{\alpha_n}$, $0\leq \alpha_i\leq n-i$, are $R$-linearly independent. It follows that
the subring $R[A_1,\dots,A_n]$ of $M_{n!}(R)$ is a universal splitting ring for $f$ and $f(Z)=(Z-A_1)\cdots (Z-A_n)$ is
a universal factorization of $f$. In the special case when $R=F$ is a field and $f$ is an irreducible and separable polynomial in $F[Z]$ with Galois group $S_n$, then $F[A_1,\dots,A_n]$
is a matrix realization of the splitting field of $f$ over $F$.

Our construction of $A_1,\dots,A_n$ is recursive in nature. It turns out that all non-zero entries of $A_1,\dots,A_n$
are equal, up to a sign, to the coefficients of $f$. This is entirely analogous to what happens to the companion matrix $C_f\in M_n(R)$
of $f$, a single universal root of $f$, although the simultaneous requirements for $A_1,\dots,A_n$ are substantially harder
to meet. We demonstrate the use of our recursive procedure by explicitly displaying $A_1,\dots,A_n$ for small values of $n$.

A key ingredient in the construction of $A_1,\dots,A_n$ is the following property of~$C_f$. If $B\in R[C_f]$ then
\begin{equation}\label{fh}
B=([B]\; C_f[B]\dots\; C_f^{n-1}[B]),
\end{equation}
where $[B]$ is the column vector of $R^n$ formed by the coordinates of $B$ relative to the $R$-basis $1,C_f,\dots,C_f^{n-1}$ of $R[C_f]$.
Property (\ref{fh}) was used in \cite{GS} to give a \emph{closed} formula for the product of any two elements belonging to simple
integral extension of $R$. This product arises often in field theory, when adjoining a single root to an
irreducible polynomial, and one is then forced to resort to the division algorithm for its computation.
In contrast, \cite{GS} furnishes the first \emph{closed} formula for this frequently encountered product.

Property (\ref{fh}) was also used in \cite{GS2} to study a wide range of features possessed by the subalgebra $A$ of $M_n(S)$, $S$ a commutative ring with $1\neq 0$, generated by two companion matrices to polynomials $g$ and $h$ of degree $n$ over $S$. For instance, if $S=\Z$ it is shown in
\cite{GS2} that $A$ is free of rank $n^2$ if and only if the resultant $R(g,h)\neq 0$, in which case the finite index
$$
[M_n(\Z):A]=|R(g,h)^{n-1}|.
$$

\section{A new set of generators for $I_f$}

We keep the above notation and assume until further notice that $R$ is an arbitrary non-zero ring with identity.

Corresponding to any transposition $(i,j)\in S_n$ there is an
$R$-linear operator $\Delta_{i,j}:R[X_1,\dots,X_n]\to
R[X_1,\dots,X_n]$ given by
$$
(\Delta_{(i,j)}
g)(X_1,\dots,,X_n)=\frac{g^{(i,j)}(X_1,\dots,X_n)-g(X_1,\dots,X_n)}{X_{j}-X_i}.
$$
%namely the quotient of dividing
%$g^{(i,j)}(X_1,\dots,X_n)-g(X_1,\dots,X_n)$ by $X_{j}-X_i$.

Observe that with this notation, we have
$$
f_1(X)=f(X_1),\, f_2=\Delta_{(1,2)} f_1,\, f_3=\Delta_{(2,3)}
f_2,\dots, f_n=\Delta_{(n-1,n)} f_{n-1}.
$$
We set
$$
I'_f=(f_1,\dots,f_n)
$$
and let $S_j^i(X_1,\dots,X_j)\in R[X_1,\dots,X_j]$ be the
sum of all monomials $X_1^{\alpha_1}\cdots X_j^{\alpha_j}$ such
that $\alpha_1+\cdots+\alpha_j=i$.

For $h_1,\dots,h_m\in R[X_1,\dots,X_n]$, the left and right ideals of $R[X_1,\dots,X_n]$ generated by $h_1,\dots,h_m$
will respectively be denoted by $l(h_1,\dots,h_m)$ and $r(h_1,\dots,h_m)$.

Furthermore, we let $g_i=\s_i-a_i$ for $1\leq i\leq n$. Note that
$$
l(g_1,\dots,g_n)+L_f[X_1,\dots,X_n]=I_f=r(g_1,\dots,g_n)+L_f[X_1,\dots,X_n].
$$

\begin{theorem}\label{e2} We have
$$
l(g_1,\dots,g_n)=l(f_1,\dots,f_n),\; r(g_1,\dots,g_n)=r(f_1,\dots,f_n)\text{ and }I_f=I'_f.
$$
Moreover,
\begin{equation}
\label{magia}
f_i=S_i^{n-(i-1)}-a_1 S_i^{n-i}+a_2 S_i^{n-(i+1)}+\cdots+(-1)^{n-(i-1)}a_{n-(i-1)},\quad 1\leq i\leq n.
\end{equation}
In particular, each $f_i$ is symmetric in $R[X_1,\dots,X_i]$ of degree $n-(i-1)$.

Furthermore, the following identity is valid for all $1\leq i\leq n$:
\begin{equation}
\label{magia2}
f_i=(\s_1-a_1)S_i^{n-i}+(-1)(\s_2-a_2)S_i^{n-(i+1)}+\cdots+(-1)^{n-i}(\s_{n-(i-1)}-a_{n-(i-1)}).
\end{equation}

\end{theorem}

\begin{proof} We begin by observing that
\begin{equation}
\label{delt} \Delta_{(j,j+1)} S_j^i=S_{j+1}^{i-1}.
\end{equation}
It is clear that (\ref{magia}) holds when $i=1$. Beginning with
this case, successively applying
$\Delta_{(1,2)},\dots,\Delta_{(n-1,n)}$, and making use of
(\ref{delt}) yields (\ref{magia}) for all $1\leq i\leq n$.

As is well-known, the following identity holds in $R[X_1,\dots,X_n][Z]$:
$$
(Z-X_1)\cdots (Z-X_n)=Z^n-\sigma_1 Z^{n-1}+\sigma_2 Z^{n-2}+\cdots+(-1)^{n}\sigma_n,
$$
whence
\begin{equation}
\label{gener}
0=X_1^n-\sigma_1 X_1^{n-1}+\sigma_2 X_1^{n-2}+\cdots+(-1)^{n}\sigma_n.
\end{equation}
Successively applying $\Delta_{(1,2)},\dots,\Delta_{(n-1,n)}$
yields
\begin{equation}
\label{gendel}
0=S_i^{n-(i-1)}-\s_1 S_i^{n-i}+\s_2 S_i^{n-(i+1)}+\cdots+(-1)^{n-(i-1)}\s_{n-(i-1)},\quad 1\leq i\leq n.
\end{equation}
Subtracting (\ref{gendel}) from (\ref{magia}) we obtain (\ref{magia2}).
The latter not only gives the inclusions
$$
\ell(f_1,\dots,f_n)\subseteq \ell(g_1,\dots,g_n),\; r(f_1,\dots,f_n)\subseteq r(g_1,\dots,g_n)\text{ and }I_f\subseteq I'_f,
$$
but reading it backwards from $i=n$ down to $i=1$ yields the reverse inclusions.
\end{proof}

As an illustration of Theorem \ref{e2}, when $n=4$ we have
$$
f_1=(\s_1-a_1)X_1^3-(\s_2-a_2)X_1^2+(\s_3-a_3)X_1-(\s_4-a_4),
$$
$$
f_2=(\s_1-a_1)(X_1^2+X_2^2+X_1X_2)-(\s_2-a_2)(X_1+X_2)+(\s_3-a_3),
$$
$$
f_3=(\s_1-a_1)(X_1+X_2+X_3)-(\s_2-a_2),
$$
$$
f_4=\s_1-a_1,
$$
as well as its alternative version
$$
f_1=X_1^4-a_1X_1^3+a_2X_1^2-a_3X_1+a_4,
$$
$$
f_2=X_1^3+X_2^3+X_1X_2^2+X_2X_1^2-a_1(X_1^2+X_2^2+X_1X_2)+a_2(X_1+X_2)-a_3,
$$
$$
f_3=X_1^2+X_2^2+X_3^2+X_1X_2+X_2X_3+X_1X_3-a_1(X_1+X_2+X_3)+a_2,
$$
$$
f_4=X_1+X_2+X_3+X_4-a_1.
$$

\section{$R_f$ is a free module when $R$ is non-commutative}\label{se1}

Recalling the notation used in the Introduction, we have the following basic result.

%As a matter of notation, if $u:S_1\to S_2$ is a ring homomorphism, the corresponding ring homomorphism $S_1[Z]\to S_2[Z]$
%will often be denoted by $u$ as well.

\begin{lemma}\label{mismo} The universal splitting ring for $f$ over $R$ is isomorphic, as ring and $T_f$-module,
to the universal splitting ring for $\pi(f)$ over $T_f$.
\end{lemma}

\begin{proof} The projection $\pi:R\to T_f$ gives rise to the epimorphisms $R[Z]\to T_f[Z]$
and $R[X_1,\dots,X_n]\to T_f[X_1,\dots,X_n]$, also denoted by $\pi$. Set
$$
R'=T_f,\; f'=\pi(f)\in R'[Z]
$$
as well as
$$
I'_{f'}=\pi(I_f)\unlhd R'[X_1,\dots,X_n],\; R'_{f'}=R'[X_1,\dots,X_n]/I'_{f'}.
$$
The projection $R\to R'$ induces the epimorphism $$R[X_1,\dots,X_n]\to R'[X_1,\dots,X_n]\to R'_{f'}.$$ Since $I_f$ is in the kernel,
we obtain an epimorphism $R_f\to R'_{f'}$. On the other hand, $L_f$ is in the kernel of $R\to R[X_1,\dots,X_n]\to R_f$, yielding a homomorphism $R'\to R_f$, which can
be extended to an epimorphism $R'[X_1,\dots,X_n]\to R_f$ with $I'_{f'}$ in its kernel. This produces an epimorphism $R'_{f'}\to R_f$,
inverse of $R_f\to R'_{f'}$.
\end{proof}

\begin{theorem}\label{q1} The ideals $L_f$ and $M_f$ are equal.
Moreover, $R_f$ is a free $T_f$-module with basis $r_1^{\al_1}\cdots r_n^{\al_n}$, where $r_i=X_i+I_f$ and $0\leq \al_i\leq
n-i$ for all $1\leq i\leq n$.
\end{theorem}

\begin{proof} If $L_f=R$ there is nothing to do, so we may suppose that $L_f$ is a proper ideal.

By Lemma \ref{mismo} we may replace $R$ by $T_f$ and assume
that the coefficients of $f$ are central in $R$. We need to show that $M_f=0$ and
$r_1^{\al_1}\cdots r_n^{\al_n}$, $0\leq \al_i\leq
n-i$, is an $R$-basis of~$R_f$. This is a well-known result when $R$ is commutative, and we proceed to indicate
how to derive it under the weaker hypothesis that $f$ has central coefficients in $R$.

Following a method that essentially goes back to Kronecker and proceeds by successive single root adjunctions (see \cite{PZ}
for details when $R$ is commutative), we may construct a ring $S$ containing $R$ as subring, with $1_R$ being
the identity of $S$, and elements $s_1,\dots,s_n$ of $S$ such that:

$\bullet$ $s_1,\dots,s_n$ commute with each other and with every element of $R$.

$\bullet$ $f(Z)=(Z-s_1)\cdots (Z-s_n)$ holds in $S[Z]$.

$\bullet$ $s_1^{\al_1}\cdots s_n^{\al_n}$, $0\leq \al_i\leq
n-i$, is an $R$-basis of $S$.

Let $\Omega:R[X_1,\dots,X_n]\to S$ be the ring epimorphism extending the inclusion $j:R\hookrightarrow S$ and
satisfying $X_i\to s_i$.  Then $I_f\subseteq\ker(\Omega)$, so $M_f\subseteq\ker (j)=(0)$.

Since $I_f\subseteq \ker(\Omega)$, we infer that $\Omega$ induces an epimorphism $\Psi:R_f\to S$ as rings and $R$-modules.
Thus $r_1^{\al_1}\cdots r_n^{\al_n}$, $0\leq \al_i\leq n-i$, are $R$-linearly independent, since so are their images under $\Psi$,
namely $s_1^{\al_1}\cdots s_n^{\al_n}$, $0\leq \al_i\leq n-i$.

On the other hand, we have
$$
f(Z)=(Z-r_1)\cdots (Z-r_n)\in R_f[Z]
$$
and
$$
R_f=R[r_1,\dots,r_n].
$$
Therefore $R_f$ is $R$-spanned by all $r_1^{\al_1}\cdots r_n^{\al_n}$, $0\leq
\al_i\leq n-i$. This is because each $r_i$ is annihilated by the monic polynomial
$(Z-r_i)\cdots (Z-r_n)\in R[r_1,\dots,r_{i-1}][Z]$ of degree $n-(i-1)$.
\end{proof}

%\begin{note}{\rm Theorem \ref{q1} and its proof remain valid when $L_f=R$ in which case
%care must be taken when interpreting our assertions. The case $L_f=R$ is certainly possible,
%for instance when $f$ has at least one non-central coefficient and $R$ is a simple ring.}
%\end{note}

In view of Lemma \ref{mismo} there is no loss of generality when studying $R_f$ in assuming
that all coefficients of $f$ are central in $R$, and we will make this assumption \emph{for the remainder of the paper}.

In light of Theorems \ref{e2} and \ref{q1} we have the following result.

\begin{cor} For any $1\leq i\leq n$, the minimal polynomial of $r_i$ over $R[r_1,\dots,r_{i-1}]$ is $f_i(r_1,\dots,r_{i-1},Z)\in R[r_1,\dots,r_{i-1}][Z]$, as described in (\ref{defensor})-(\ref{defensor3}) or (\ref{magia}).
\end{cor}

\section{$R_f$ viewed as ring of $R$-linear operators}

Here we use $f_1,\dots,f_n$ to define, from scratch, a ring of $R$-linear operators, which turns out to be isomorphic to $R_f$.
For this purpose, we view $R[X_1,\dots,X_n]$ as a right $R$-module, noting that $R$ acts on it
via $R$-endomorphisms by left multiplication. Let $V$ be the $R$-submodule of $R[X_1,\dots,X_n]$
spanned by all monomials $X_1^{\al_1}\cdots X_n^{\al_n}$, where $0\leq
\al_i\leq n-i$ for every $1\leq i\leq n$. Let $R\langle Y_1,\dots,Y_n\rangle$ be the ring
of polynomials in the non-commuting variables $Y_1,\dots,Y_n$ over $R$.
We inductively define $R$-linear endomorphisms
$L_{Y_1},\dots, L_{Y_n}$ of $V$ as follows. We first~let
$$
L_{Y_1}X_1^{\al_1}\cdots X_n^{\al_n}=X_1^{\al_1+1}\cdots
X_n^{\al_n},\quad \text{ if }\al_1<n-1.
$$
Noting that $X_1^n-f_1(X_1)=a_{1}X_1^{n-1}-a_2 X_1^{n-2}+\cdots+(-1)^{n-1}a_{n}$, we next define
$L_{Y_1}X_1^{n-1}X_2^{\al_2}\cdots X_n^{\al_n}$ to be equal to
$$
a_1X_1^{n-1}X_2^{\al_2}\cdots X_n^{\al_n}-
a_2X_1^{n-2}X_2^{\al_2}\cdots X_n^{\al_n}+\cdots+(-1)^{n-1}a_nX_2^{\al_2}\cdots X_n^{\al_n}.
$$
Suppose we have defined $L_{Y_1},\dots, L_{Y_{i-1}}\in \mathrm{End}_R(V)$ for some $1<i\leq n$.
This gives rise to a unique ring homomorphism $L^{i-1}:R\langle Y_1,\dots,Y_{i-1}\rangle\to \mathrm{End}_R(V)$,
that extends the action of $R$ on $V$ and satisfies $Y_j\mapsto L_{Y_j}$ for all $1\leq j\leq i-1$.
Now, it follows from (\ref{magia}) that
$$
X_i^{n-(i-1)}-f_i(X_1,\dots,X_i)=h_{n-i}(X_1,\dots,X_{i-1})X_i^{n-i}+\cdots+h_0(X_1,\dots,X_{i-1})
$$
for unique $h_{n-i},\dots,h_0\in R[X_1,\dots,X_{i-1}]$, and we let
$$
L_{Y_{i}}X_1^{\al_1}\cdots X_i^{\al_i}\cdots X_n^{\al_n}=X_1^{\al_1}\cdots X_{i}^{\al_i+1}\cdots X_n^{\al_n},\quad
\text{ if }\al_i<n-i,
$$
while $L_{Y_{i}}X_1^{\al_1}\cdots X_i^{n-i}\cdots X_n^{\al_n}$ is defined to be
$$
L^{i-1}_{h_{n-i}(Y_1,\dots,Y_{i-1})}X_1^{\al_1}\cdots X_i^{n-i}\cdots X_n^{\al_n}+\cdots+
L^{i-1}_{h_{0}(Y_1,\dots,Y_{i-1})}X_1^{\al_1}\cdots X_i^{0}\cdots X_n^{\al_n}.
$$
\begin{theorem}\label{e3} The operators $L_{Y_1},\dots, L_{Y_n}$ commute with each other and
and with the action of $R$ on $V$ by left multiplication. The corresponding ring homomorphism $R[X_1,\dots,X_n]\to\mathrm{End}_R(V)$, satisfying $X_i\to L_{Y_i}$, has kernel $I_f$ and, consequently, $R_f\cong R[L_{Y_1},\dots,L_{Y_n}]$.
\end{theorem}

\begin{proof} Let $\ell:R_f\to\mathrm{End}_R(R_f)$ be the regular representation,
where $R_f$ is viewed as a right $R$-module and $R_f$ acts on
itself by left multiplication. The action of $r_1,\dots,r_n$ on
the basis vectors $r_1^{\al_1}\cdots r_n^{\al_n}$, $0\leq
\al_i\leq n-i$, can be computed using that
$r_i^{n-(i-1)}-f_i(r_1,\dots,r_i)$ is an $R$-linear combination of
$r_1^{\be_1}\cdots r_i^{\be_i}$, with~$\be_i\leq n-i$. The
isomorphism of right $R$-modules $R_f\to V$ given by
$r_1^{\al_1}\cdots r_n^{\al_n} \to X_1^{\al_1}\cdots X_n^{\al_n}$
gives rise to a ring isomorphism $\mathrm{End}_R(R_f)\to
\mathrm{End}_R(V)$, and $L_{Y_1},\dots,L_{Y_n}$ correspond to
$\ell_{r_1},\dots,\ell_{r_n}$ under this isomorphism. In
particular, $L_{Y_1},\dots,L_{Y_n}$ commute with each other and
with the action of $R$ on $V$, which gives rise to the stated ring
homomorphism $R[X_1,\dots,X_n]\to \mathrm{End}_R(V)$.

On the other hand, the factorization $f(Z)=(Z-r_1)\cdots (Z-r_n)$
in $R_f[Z]$ produces the factorization $f(Z)=(Z-\ell_{r_1})\cdots
(Z-\ell_{r_n})$ in $\mathrm{End}_R(R_f)[Z]$, via the regular
representation, which in turn gives, via
$\mathrm{End}_R(R_f)\to \mathrm{End}_R(V)$, the factorization
$f(Z)=(Z-L_{Y_1})\cdots (Z-L_{Y_n})$ in $\mathrm{End}_R(V)[Z]$.
This implies that $I_f$ is included in the kernel of
$R[X_1,\dots,X_n]\to \mathrm{End}_R(V)$. That $I_f$ is actually
the kernel is equivalent to $L_{Y_1}^{\al_1}\cdots
L_{Y_n}^{\al_n}$, $0\leq \al_i\leq n-i$, being linearly
independent over~$R$. This can be seen by applying these operators
to $1\in V$.
\end{proof}

\section{A matrix realization of $R_f$}\label{matai}

Here we obtain a matrix realization of $R_f$ via matrices $A_1,\dots,A_n\in M_{n!}(R)$ corresponding
to the $R$-linear operators $\ell_{r_1}\dots\ell_{r_n}$ of $R_f$ arising from the left regular representation
$\ell:R_f\to\mathrm{End}_R(R_f)$, where $R_f$ is viewed as a right $R$-module.

For notational simplicity it will be convenient to write
$$
f(Z)=Z^n+b_{n-1}Z^{n-1}+\cdots+b_1Z+b_0,
$$
where $b_0,b_1,\dots,b_{n-1}\in R$ are still supposed to be central in $R$. Let
$$
C_f=\left(%
\begin{array}{ccccc}
  0 & 0 & \cdots & 0 & -b_0 \\
  1 & 0 & \cdots & 0 & -b_1 \\
  0 & 1 & \cdots & 0 & -b_2 \\
  \vdots & \vdots & \cdots & \vdots & \vdots \\
  0 & 0 & \cdots & 1 & -b_{n-1} \\
\end{array}%
\right)\in M_n(R)
$$
be the companion matrix to $f$. It will be useful to know the appearance of the elements of $R[C_f]$.

For this purpose, given $g\in R[Z]$ set $\widetilde{g}=g+(f)\in R[Z]/(f)$ and let $[g]=[\widetilde{g}]$
be the column vector in $R^n$ formed by the coordinates
of $\widetilde{g}$ relative to the $R$-basis $\widetilde{1},\widetilde{Z},\dots, \widetilde{Z^{n-1}}$ of $R[Z]/(f)$.

\begin{lemma}\label{ferhec} For $g\in R[Z]$ we have
\begin{equation}\label{otri}
g(C_f)= ([g]\; C_f[g]\dots\; C_f^{n-1}[g])=([g]\; [Zg]\dots\; [Z^{n-1} g]).
\end{equation}
\end{lemma}

\begin{proof} Let $h=c_{n-1}Z^{n-1}+\cdots+c_1Z+c_0\in R[Z]$ be the unique polynomial satisfying $g\equiv h\mod (f)$.
It clearly suffices to prove the result for $h$ instead of $g$. Now
$$h(C_f)e_1=[h],$$
so the first columns of the left and right hand sides of (\ref{otri}) are equal. Moreover, for $1<i\leq n$ we have
$$
h(C_f)e_i=h(C_f)C_f^{i-1}e_1=C_f^{i-1}h(C_f)e_1=C_f^{i-1}[h],
$$
which proves both equalities, provided we agree that
$$
C_f^j[p]=[X^j p],\quad p\in R[X], j\geq 0.
$$
This is obvious since the $R$-linear endomorphism of $R[Z]/(f)$ given by multiplication by $\widetilde{X}$ has matrix $C_f$, whence
$$C_f[\widetilde{p}]=[\widetilde{Xp}],\quad p\in R[X].$$
\end{proof}

There are exactly $n$ monic polynomials $g\in R[X]$ of degree $<n$ such that all coefficients of $g(C_f)$ are either 0 or equal to an actual coefficient of $f$, up to a sign. Moreover, for such $g$ the appearance of $g(C_f)$, as given in (\ref{otri}), can be made substantially more explicit.

We proceed to define these polynomials. For this purpose, given $g\in R[Z]$ we define
$$
g^{[0]}(Z)=\frac{g(Z)-g(0)}{Z},\, g^{[1]}(Z)=\frac{g^{[0]}(Z)-g^{[0]}(0)}{Z},\, g^{[2]}(Z)=\frac{g^{[1]}(Z)-g^{[1]}(0)}{Z},\dots.
$$
Thus, if $f(Z)=Z^m+c_{m-1}Z^{m-1}+\cdots+c_1Z+c_0$ then
$$
g^{[0]}(Z)=Z^{m-1}+c_{m-1}Z^{m-2}+\cdots+c_2Z+c_1,\dots,
$$
$$
g^{[m-2]}(Z)=Z+c_{m-1},\, g^{[m-1]}(Z)=1,\, g^{[j]}(Z)=0,\quad j\geq m.
$$
A careful examination of (\ref{otri}) together with the fundamental relation
$$
f(C_f)=0
$$
reveals the exact appearance of $g(C_f)$ for all polynomials $g=f^{[j]}$, $j\geq 0$.
In particular, the coefficients of all such $g(C_f)$ are either 0 or equal to a coefficient of~$f$, up to a sign. We have
\begin{equation}\label{exacto}
f^{[0]}(C_f)=\left(%
\begin{array}{ccccc}
  b_1 & -b_0 & 0 & \cdots & 0 \\
  b_2 & 0 & -b_0 & \ddots & \vdots \\
  \vdots & \vdots & \ddots & \ddots & 0 \\
  b_{n-1} & \vdots & \vdots & \ddots & -b_0 \\
  1 & 0 & 0 & 0 & 0 \\
\end{array}%
\right),
\end{equation}
\begin{equation}\label{exacto2}
f^{[1]}(C_f)=\left(%
\begin{array}{ccccccc}
  b_2 & 0 & -b_0 & 0 & \cdots & 0 & 0\\
  b_3 & b_2 & -b_1 & -b_0 & \ddots & \vdots & \vdots \\
  b_4 & b_3 & 0 & -b_1 & \ddots & \ddots & \vdots \\
  \vdots & \vdots & \vdots & \vdots & \ddots & -b_0  & 0 \\
  b_{n-1} & b_{n-2} & \vdots & \vdots & \vdots & -b_1 & -b_0 \\
  1 & b_{n-1} & \vdots & \vdots & \vdots & 0 &  -b_1 \\
  0 & 1 & 0 & 0 & \cdots & 0 & 0 \\
\end{array}%
\right),\dots,
\end{equation}
\begin{equation}\label{exacto3}
f^{[n-2]}(C_f)=\left(%
\begin{array}{ccccc}
  b_{n-1} & 0 & \cdots & 0 & -b_0 \\
  1 & b_{n-1} & \vdots & \vdots & -b_1 \\
  0 & 1 & \ddots & \vdots & \vdots \\
  \vdots & \vdots & \ddots & b_{n-1} & -b_{n-2} \\
  0 & 0 & \cdots & 1 & 0 \\
\end{array}%
\right),
\end{equation}
\begin{equation}\label{exacto4}
f^{[n-1]}(C_f)=I_n\text{ and }f^{[j]}(C_f)=0_n,\quad j\geq n.
\end{equation}

We next define a total order on the basis $r_1^{\al_1}\cdots r_n^{\al_n}$, $0\leq \al_i\leq n-i$. If $n=1$ there is only one possible order. If $n>1$ let $s_1$ be the sequence
$1,r_1,\dots,r_1^{n-1}$; $s_2$ the sequence $s_1,s_1r_2,\dots,s_1r_2^{n-2}$; 
$s_3$ the sequence $s_2,s_2r_3,\dots,s_2r_3^{n-3}$; and so on. We order $r_1^{\al_1}\cdots r_n^{\al_n}$ according to the sequence  $s_{n-1}$.

Suppose $n>1$ and let $S=R[r_1]$. Then $R_f$ is a free $S$-module with basis $r_2^{\al_1}\cdots r_n^{\al_n}$, $0\leq \al_i\leq n-i$,
with order inherited from the above. In fact, if $$g(Z)=f_2(r_1,Z)=\frac{f(Z)-f(r_1)}{Z-r_1}\in S[Z],$$ then $R_f=S_g$ is the universal
splitting ring for $g$ over $S$. Note that
\begin{equation}\label{gjota}
g(Z)=Z^{n-1}+f^{[n-2]}(r_1)Z^{n-2}+\cdots+f^{[1]}(r_1)Z+f^{[0]}(r_1).
\end{equation}
%We are ready to indicate how $A_1,\dots,A_n$ can be recursively constructed.

\begin{theorem}\label{matex} The matrices $A_1,\dots,A_n\in M_{n!}(S)$ can be recursively constructed as follows.

(1)
$$
A_1=C_f\oplus\cdots\oplus C_f,\quad (n-1)!\text{ summands}.
$$

(2) In particular, if $n=1$ then $A_1=(-b_0)$.

(3) Suppose $n>1$. Let $B_2,\dots,B_n\in M_{(n-1)!}(S)$ be the matrices corresponding
to the $S$-linear operators $\ell_{r_2}\dots\ell_{r_n}$ of $R_f=S_g$ relative to the basis $r_2^{\al_1}\cdots r_n^{\al_n}$, $0\leq \al_i\leq n-i$,
ordered as indicated above. Then for each $2\leq i\leq n$, $A_i$ is obtained from $B_i$ by replacing each entry, necessarily of the form $\pm f^{[j]}(r_1)\in S$, $j\geq 0$, by $\pm f^{[j]}(C_f)\in M_n(R)$, where this matrix is explicitly given in (\ref{exacto})-(\ref{exacto4}).

(4) In particular, every non-zero entry of $A_1,\dots,A_n$ is equal to a coefficient of~$f$, up to a sign.
\end{theorem}

\begin{proof} By induction on $n$. The result is clearly true when $n=1$. Suppose that $n>1$ and let $B_2,\dots,B_n\in M_{(n-1)!}(S)$ be the matrices corresponding to the $S$-linear operators $\ell_{r_2}\dots\ell_{r_n}$ of $R_f=S_g$ relative to the basis $r_2^{\al_1}\cdots r_n^{\al_n}$, $0\leq \al_i\leq n-i$, ordered as indicated above. By inductive assumption every non-zero entry of $B_2,\dots,B_n$ is equal to a coefficient of~$g$, up to a sign. By (\ref{gjota}) the coefficients of $g$ are $f^{[j]}(r_1)$, $0\leq j\leq n-1$, and we know that $f^{[j]}=0$ for $j\leq n$. Since the matrix of the $R$-linear operator $\ell_{r_1}$ of $R[r_1]$ relative to the basis $1,r_1,\dots,r_1^{n-1}$ is $C_f$, it follows that
the matrix of $\ell_{f^{[j]}(r_1)}=f^{[j]}(\ell_{r_1})$ is equal to $f^{[j]}(C_f)$, $j\geq 0$. We infer that each $A_i$, $2\leq i\leq n$, is obtained from $B_i$ by replacing each entry $\pm f^{[j]}(r_1)$, $j\geq 0$, by $\pm f^{[j]}(C_f)\in M_n(R)$, where this matrix is explicitly given in (\ref{exacto})-(\ref{exacto4}).
\end{proof}

As an illustration of Theorem \ref{matex}, let us compute the desired matrices when $n=2$ from the case $n=1$, and then
proceed onwards to the case $n=3$ from the case $n=2$. If $n=1$ we have $A_1=(-b_0)$. Moreover, if $n=2$ then
$$
A_1=\left(
      \begin{array}{cc}
        0 & -b_0 \\
        1 & -b_1 \\
      \end{array}
    \right),
$$
with $g(Z)=Z+(r_1+b_1)$ by (\ref{gjota}). Writing this in the form $g(Z)=Z+c_0$ and going back to the case $n=1$ we get $B_2=(-c_0)$,
which results in
$$
A_2=-(C_f+b_1)=\left(
      \begin{array}{cc}
        -b_1 & b_0 \\
        -1 & 0 \\
      \end{array}
    \right).
$$
Furthermore, if $n=3$ then
$$
A_1=C_f\oplus C_f=\left(
    \begin{array}{cccccc}
      0 & 0 & -b_0 & 0 & 0 & 0 \\
      1 & 0 & -b_1 & 0 & 0 & 0 \\
      0 & 1 & -b_2 & 0 & 0 & 0 \\
      0 & 0 & 0 & 0 & 0 & -b_0 \\
      0 & 0 & 0 & 1 & 0 & -b_1 \\
      0 & 0 & 0 & 0 & 1 & -b_2 \\
    \end{array}
  \right),
$$
with $g(Z)=Z^2+(r_1+b_2)Z+(r_1^2+b_2r_1+b_1)$ by (\ref{gjota}). Writing this in the form $g(Z)=Z^2+c_1Z+c_0$ and going back to the case $n=2$ we get
$$
B_2=\left(
      \begin{array}{cc}
        0 & -c_0 \\
        1 & -c_1 \\
      \end{array}
    \right),\; B_3=\left(
      \begin{array}{cc}
        -c_1 & c_0 \\
        -1 & 0 \\
      \end{array}
    \right),
$$
which, thanks to (\ref{exacto})-(\ref{exacto4}), results in
$$
A_2=\left(
    \begin{array}{cccccc}
      0 & 0 & 0 & -b_1 & b_0 & 0 \\
      0 & 0 & 0 & -b_2 & 0 & b_0 \\
      0 & 0 & 0 & -1 & 0 & 0 \\
      1 & 0 & 0 & -b_2 & 0 & b_0 \\
      0 & 1 & 0 & -1 & -b_2 & b_1 \\
      0 & 0 & 1 & 0 & -1 & 0 \\
    \end{array}
  \right),\;
A_3=\left(
    \begin{array}{cccccc}
      -b_2 & 0 & b_0 & b_1 & -b_0 & 0 \\
      -1 & -b_2 & b_1 & b_2 & 0 & -b_0 \\
      0 & -1 & 0 & 1 & 0 & 0 \\
      -1 & 0 & 0 & 0 & 0 & 0 \\
      0 & -1 & 0 & 0 & 0 & 0 \\
      0 & 0 & -1 & 0 & 0 & 0 \\
    \end{array}
  \right).
$$
Here $A_1,A_2,A_3$ commute with each other and with every element of $R$,
$$A_1+A_2+A_3=-b_2,\, A_1A_2+A_1A_3+A_2A_3=b_1,\, A_1A_2A_3=-b_0,
$$
and $1,A_1,A_1^2,A_2,A_1A_2,A_1^2A_2$ are $R$-linearly independent. Thus,
$R_f\cong R[A_1,A_2,A_3]$.

It is clear how to use the case $n=3$ and (\ref{exacto})-(\ref{gjota}) to to obtain the case $n=4$. The process can be continued indefinitely.

\section{Uniqueness of the roots of $f$}\label{aut}

It should be borne in mind that $r_1,\dots,r_n$ need not be the only roots of $f$ in $R_f$.
Indeed, observe that if $t_1,\dots,t_n\in R_f$ the map $p(r_1,\dots,r_n)\mapsto p(t_1,\dots,t_n)$, where $p(X_1,\dots,X_n)\in R[X_1,\dots,X_n]$,
is an automorphism of $R_f$ over $R$ if and only if $t_1,\dots,t_n$ commute with each other and with every element of $R$,
the factorization $f(Z)=(Z-t_1)\cdots (Z-t_n)$ holds in $R_f[Z]$, and $t_1^{\al_1}\cdots t_n^{\al_n}$, $0\leq \al_i\leq n-i$, form an $R$-basis of $R_f$.

Let us view $S_n$ as a subgroup of
$\mathrm{Aut}(R[X_1,\dots,X_n]/R)$. Since $S_n$ preserves~$I_f$, every $\sigma\in S_n$ gives rise to an automorphism
$\widetilde{\s}\in \mathrm{Aut}(R[X_1,\dots,X_n]/I_f)$ that fixes $R$ pointwise, i.e., an automorphism of $R_f$ over $R$.
The map $\s\mapsto \widetilde{\s}$ is a group homomorphism $\Theta:S_n\to\mathrm{Aut}(R_f/R)$. We assume for the remainder of this section that $n>2$. It then follows easily from
Theorem \ref{q1} that $\Theta$ is injective.

The point is that the automorphism group of $R_f$ over $R$ need not reduce to $S_n$. As a matter of fact, let $U$ be the group of central units of $R$. Suppose first that $f(Z)=Z^n$. Then $U$ becomes a subgroup of $\mathrm{Aut}(R_f/R)$ by letting $t_i=ur_i$, $u\in U$,
and $U\cap S_n$ is trivial. More generally, suppose $n=dm$ and that all coefficients $a_i$ of $f$ such that $i\not\equiv 0\mod d$ are equal to 0. Let $U_d$ be the subgroup of $U$ of all $u$ satisfying $u^d=1$ and let $t_i=ur_i$, $u\in U$. Then
$$\s_i(t_1,\dots,t_n)=u^i\s_i(r_1,\dots,r_n)=\s_i(r_1,\dots,r_n),\quad 1\leq i\leq n,$$
so $U_d$ becomes a subgroup of $\mathrm{Aut}(R_f/R)$ and $U_d\cap S_n$
is trivial.

It may be of interest to determine $\mathrm{Aut}(R_f/R)$ and, in particular, when this reduces to $S_n$.
%We wonder if this will be the case when the coefficients of $f$ are algebraically independent over $R$.

%It may also be worthwhile to investigate the uniqueness of the matrices $A_1,\dots,A_n$ constructed in \S\ref{matai},
%in the following sense. Suppose $D_1,\dots,D_n\in M_{n!}(R)$ satisfy the following conditions: $D_1,\dots,D_n$ commute with each other and with every element of $R$; $f(Z)=(Z-D_1)\cdots (Z-D_n)$; $D_1^{\alpha_1}\cdots,D_n^{\alpha_n}$, $0\leq \alpha_i\leq n-i$, are $R$-linearly independent;
%all non-zero entries of $D_1,\dots,D_n$ are equal, up to a sign, to a coefficient of $f$. Then $D_1,\cdots,D_n$ is a permutation of $A_1,\dots,A_n$.
%On the same vein, we wonder if $C_f$ is the only matrix $D$ in
%$M_n(R)$ such that $1,D,\dots,D^{n-1}$ are $R$-linearly
%independent, $f(D)=0$, and every non-zero entry of $D$ is equal,
%up to a sign, to a coefficient of $f$.
% The above fails if D is nilpotent. Even the first question fails due to transposition or conjugation by permutation in $M_{n!}(R)$. If the coefficients of $f$ are alg independent, may be these are the only changes.

%==================================================


\begin{thebibliography}{FH}
%==================================================
\bibitem[AG]{AG} M. Auslander and O. Goldman, \emph{The Brauer group of a commutative ring}, Trans. Amer. Math. Soc.  97  (1960)
367-–409.

\bibitem[Bo]{Bo} N. Bourbaki,  \emph{$\acute{E}l\acute{e}ments$ de math$\acute{e}$matique. Alg$\grave{e}$bre. Chapitres 4
$\grave{a}$ 7}, Masson, Paris, 1981.

\bibitem[Ba]{Ba} A. Barnard, \emph{Commutative rings with operators (Galois theory and
ramification)}, Proc. London Math. Soc. (3)  28  (1974) 274-–290.

\bibitem[BS]{BS} M. Bhargawa and M. Satriano, \emph{On a notion of ``Galois closure" for extensions of rings},
J. Eur. Math. Soc. 16 (2014) 1881--1913.

\bibitem[EL]{EL} T. Ekedahl and D. Laksov, \emph{Splitting algebras, symmetric functions
and Galois theory}, J. Algebra Appl. 4 (2005) 59-–75.

\bibitem[G]{G} A. Gioia, \emph{On the Galois closure of commutative algebras}, Ph.D. thesis, Universit${\acute{e}}$ Bordeaux~I, 2013.


\bibitem[GS]{GS} N. Guersenzvaig and F. Szechtman, \emph{A closed formula for the product in simple integral
extensions}, Linear Algebra Appl. 430 (2009) 2464–-2466.

\bibitem[GS2]{GS2} N. Guersenzvaig and F. Szechtman, \emph{Subalgebras of matrix algebras generated by companion matrices}, Linear Algebra Appl.
432 (2010) 2691–-2700.

\bibitem[N]{N} T. Nagahara, \emph{On separable polynomials over a commutative ring. II.}, Math. J. Okayama Univ.  15  (1971/72)
149-–162.

\bibitem[PZ]{PZ} M. Pohst and H. Zassenhaus, \emph{Algorithmic algebraic number
theory}, Encyclopedia of Mathematics and its Applications,
Cambridge University Press, Cambridge, 1989.

\bibitem[W]{W} S. Wang, \emph{Splitting ring of a monic separable
polynomial}, Pacific J. Math.  75  (1978) 293–-296.


\end{thebibliography}
\end{document}